\documentclass[11 pt]{amsart}
\usepackage[utf8]{inputenc}
\usepackage{article-preamble}
\usepackage[textsize = small, color = pink]
{todonotes}
\title{A family of accumulation points of non-free rational numbers}

\author{Christopher (Bobby) Buyalos}
\address[C. Buyalos]{University of Chicago, United States}
\email{\href{mailto:cbuyalos@uchicago.edu}{cbuyalos@uchicago.edu}}

\author{Jayden Thadani}
\address[J. Thadani]{Harvey Mudd College, United States}
\email{\href{mailto:jthadani@hmc.edu}{jthadani@hmc.edu}}

\author{Xinbei Wang}
\address[X. Wang]{Brown University, United States}
\email{\href{mailto:xinbei_wang@brown.edu}{xinbei\_wang@brown.edu}}

\author{Bradley Zykoski}
\address[B. Zykoski]{Northwestern University, United States}
\email{\href{mailto:zykoski@northwestern.edu}{zykoski@northwestern.edu}}

\author{Michael Zshornack (corresponding author)}
\address[M. Zshornack]{Northwestern University, United States}
\email{\href{mailto:zshornack@northwestern.edu}{zshornack@northwestern.edu}}

\begin{document}

\begin{abstract}
    For any $q\in\bbR$, let $A:=\big(\begin{smallmatrix}1 & 1\\0 & 1\end{smallmatrix}\big), B_q:=\big(\begin{smallmatrix}1 & 0\\q & 1\end{smallmatrix}\big)$ and let $G_q:=\langle A,B_q\rangle\leqslant\operatorname{SL}(2,\mathbb{R})$. Kim and Koberda conjecture that for every $q\in\mathbb{Q}\cap(-4,4)$, the group $G_q$ is not freely generated by these two matrices. We generalize work of Smilga and construct families of $q$ satisfying the conjecture that accumulate at infinitely many different points in $(-4,4)$. We give different constructions of such families, the first coming from applying tools in Diophantine geometry to certain polynomials arising in Smilga's work, the second from sums of geometric series and the last from ratios of Pell and Half-Companion Pell Numbers accumulating at $1+\sqrt{2}$.
\end{abstract}

\maketitle

\section{Introduction}
For any $q\in\bbR$, let $A:=\left(\begin{smallmatrix}1 & 1\\0 & 1\end{smallmatrix}\right), B_q:=\left(\begin{smallmatrix}1 & 0\\q & 1\end{smallmatrix}\right)$ and let
\[
G_q :=\langle A, B_q\rangle\leqslant\operatorname{SL}(2,\bbR).
\]
In general, we are far from completely understanding the dependence of the group-theoretic properties of $G_q$ on the choice of parameter $q$. When $|q|\geq 4$, a simple ping-pong argument shows that $G_q$ is both discrete and freely generated by $A$ and $B_q$, but for $|q|<4$, much less is known about the overall structure of these groups. For one, it is not known if such a freely generated example exists for rational $|q|<4$, though such an example would necessarily be indiscrete.

The freeness question for $G_q$ has been extensively studied in \cite{brenner1955}, \cite{Lyndon_Ullman_1969} and many other works. It was later explicitly conjectured in \cite{kimkoberda} that for all $q\in\bbQ$ with $|q|<4$, $G_q$ is never isomorphic to a free group of rank $2$. For concreteness, we define the set
\[
\mathcal{R}_\bbQ:=\left\{q\in\bbQ\,:\,G_q\textrm{ is not freely generated by }A\textrm{ and }B_q\right\}.
\]
We refer to elements $q\in\mathcal{R}_\bbQ$ as \textit{non-free} rational numbers. In \cite{kimkoberda}, the authors pose the following conjecture.

\begin{conjecture}
\label{kimkoberdaconjecture}
    $\mathcal{R}_\bbQ=\bbQ\cap(-4,4)$.
\end{conjecture}

Kim and Koberda give evidence for this conjecture by establishing it for all $q\in\bbQ\cap(-4,4)$ whose numerator, in least terms, is at most $27$ and not $24$, along with exhibiting positive density subsets of rationals with arbitrary fixed numerator such that the conjecture holds. Despite this progress, a number of seemingly basic properties of the set $\mathcal{R}_\bbQ$ remain mysterious. For instance, it is still unknown if $\mathcal{R}_\bbQ$ is dense in $[-4,4]$, or even if $\pm 4\in\overline{\mathcal{R}_\bbQ}$. To the authors' best understanding it is even unknown whether $\overline{\mathcal{R}_\bbQ}$ contains any interval or has positive measure.

Short of establishing this conjecture in full, it is fruitful to try and show $\mathcal{R}_\bbQ$ contains larger and larger subsets of $\bbQ\cap(-4,4)$ than previously known. For instance, in \cite{Smilga2021}, the author produces novel sequences in $\mathcal{R}_\bbQ$ which limit to ``large'' elements of $[-4,4]$. Note that, for instance, while the results in \cite{kimkoberda} provide a wealth of new examples of non-free rational numbers, a priori, these all accumulate at $0$. In practice, it is much harder to find examples of non-free rationals which accumulate at large elements of $[-4,4]$, such as $\pm 4$. Smilga's examples show that both $\pm(2+\sqrt{2})$ and $\pm\frac{3+\sqrt{5}}{2}$ are such accumulation points and, at the time of writing, are among the largest elements of $[-4,4]$ known to be contained in $\overline{\mathcal{R}_\bbQ}$.

This paper extends the methods of \cite{Smilga2021} to construct new families of elements in $\mathcal{R}_\bbQ$ which accumulate at new points in $[-4,4]$. None of the examples we produce are larger than those in \cite{Smilga2021} (our largest example is $2$) but the main result constructs sequences of non-free rational numbers accumulating at a family of rational numbers naturally arising in the work of Kim--Koberda. See Theorem~\ref{thm:main} for a more precise statement.

\begin{theorem}
\label{thm:main2}
    For nonzero integers $r,s,t$, we have that $\frac{r+t}{rst}$ is a limit of a non-constant sequence of non-free rational numbers.
\end{theorem}

\subsection{Sketch of the proof}
In \cite{Smilga2021}, the author shows that a relation of a particular form in $A$ and $B_q$ holding is equivalent to $q$ being a root of a polynomial in a certain family. The roots of the quadratic polynomials in this family can be explicitly determined, but for these roots to be rational, a certain discriminant must be a perfect square. To construct non-free rational numbers accumulating at various distinct points in $[-4,4]$, we use work of \cite{SILVERMAN2000163} to show that for some quadratic polynomials in this family, this Diophantine obstruction is described by an algebraic curve containing infinitely many integral points. Consequently, we produce many $q\in\bbQ$ satisfying \textit{half-relations} of the form in \cite{Smilga2021}. The non-free rational numbers produced depend on a tuple of parameters, $(a_1,a_2,a_3,a_4,a_5)$ which, when sent to infinity appropriately, produce the various accumulation points in Theorem~\ref{thm:main2}.

\subsection{Groups generated by opposite parabolic matrices}
A subgroup of $\operatorname{SL}(2,\bbQ)$ generated by a pair of parabolic matrices stabilizing distinct lines in $\bbR\mathbb{P}^1$ is always conjugate over $\operatorname{SL}(2,\bbR)$ to a group of the form $G_q$ for some choice of $q$. Such a conjugacy need not preserve the group being contained in $\operatorname{SL}(2,\bbQ)$ and so, more generally, one can study similar questions explored here about subgroups generated by any pair of distinct parabolic matrices with rational entries.

Recent work on questions asked by Greenberg and Shalom on commensurators of discrete subgroups of semisimple Lie groups bears a number of consequences on the structure of such groups. In \cite{brody-fisher-mj-limbeek-2023}, the authors consider the question (dubbed \textit{the Greenberg--Shalom Hypothesis}) of whether Zariski-dense discrete subgroups of semisimple Lie groups with almost dense commensurator are necessarily arithmetic lattices. This question generalizes a characterization of arithmeticity among lattices originally due to Margulis and in particular, if true, would imply that Conjecture~\ref{kimkoberdaconjecture} holds. More specifically (c.f. Theorem 4.8 in \cite{brody-fisher-mj-limbeek-2023}), if $q=\frac{r}{s}\in\bbQ\cap(-4,4)$ is written in least terms, then the Greenberg--Shalom Hypothesis implies that $G_q$ is a finite-index subgroup of the $S$-arithmetic group $\operatorname{SL}(2,\bbZ[\frac{1}{s}])$, and in particular, is not isomorphic to a free group. 

Further evidence for both the Greenberg--Shalom Hypothesis and Conjecture~\ref{kimkoberdaconjecture} is given in \cite{nybergbrodda} where the author shows, among other results, that $G_q$ is indeed finite index in the corresponding $S$-arithmetic group for $q\in\bbQ\cap(-4,4)$ whose numerator, in least terms, is at most $4$. 

\subsection{Some further questions}
We close with a number of questions that have not been recorded in previous literature on this topic, but we feel are worth further exploration.

First, the rational numbers in the statement of Theorem~\ref{thm:main2} are of note because they arise as some of the simplest types of relation non-free within a hierarchy introduced in \cite{kimkoberda}. See \S\ref{sec:step} for a precise definition but, in short, these \textit{$1$-step relation numbers} are among the elements of $\mathcal{R}_\bbQ$ for which it is particularly easy to exhibit a relation in the corresponding group. It is worth further understanding the relationship between the hierarchy of non-free numbers used in \cite{kimkoberda} and the property of being an accumulation point of $\mathcal{R}_\bbQ$. If $1$-step relation numbers are among the easiest rationals for which one can exhibit a nontrivial relation in $G_q$, then our work establishes that they are also among the easiest to exhibit as accumulation points of $\mathcal{R}_\bbQ$. Since Kim and Koberda also give a number of explicit examples of \textit{$2$-step relation numbers}, the next stage in their hierarchy, it is natural to wonder whether it is also relatively easy to show these are accumulation points. Notably, the rationals in $\bbQ\cap(-4,4)$ with numerator at most $27$ and not $24$ are all $2$-step relation numbers. Thus, we ask the following.

\begin{question}
    Are $2$-step relation numbers accumulation points of $\mathcal{R}_\bbQ$?
\end{question}

Notably, if true, then this would imply $\frac{27}{7}\approx 3.857$ is an accumulation point of non-free rational numbers and, to our best understanding, would be the largest such example known to date. 

Beyond just producing new examples of elements in $\mathcal{R}_\bbQ$, it would be meaningful progress towards Conjecture~\ref{kimkoberdaconjecture} to describe how quantitatively ``large'' this set is. While it has been previously asked whether $\mathcal{R}_\bbQ$ is dense in $[-4,4]$, the following more elementary questions seem to still be unknown.

\begin{question}
    Does $\overline{\mathcal{R}}_\bbQ$ contain any interval? Does it have positive measure?
\end{question}

Our focus on \textit{rationals} in $(-4,4)$ is in some sense superficial. The freeness question and the implications of the Greenberg--Shalom Hypothesis on $G_q$ also apply to many algebraic choices of $q\in(-4,4)$. We let
\[
\mathcal{R}_{\overline{\bbQ}}:=\left\{q\in\bbR\cap\overline{\bbQ}\,:\,G_q\textrm{ is not freely generated by }A\textrm{ and }B_q\right\}.
\]
The Greenberg--Shalom Hypothesis implies that for non-integral $q$, $q\notin\mathcal{R}_{\overline{\bbQ}}$ if and only if $q$ has a conjugate, $q^\sigma$, such that $G_{q^\sigma}$ is both free \textit{and} discrete (c.f. Theorem 4.9 in \cite{brody-fisher-mj-limbeek-2023}). Unlike the rational case, work in \cite{cjr} shows that $\mathcal{R}_{\overline{\bbQ}}$ is dense in $[-4,4]$; however, density is established using algebraic numbers of arbitrarily large degree. For any $d\geq 1$, if we set
\[
\mathcal{R}_{\overline{\bbQ}}^{(d)}:=\left\{q\in\bbR\cap\overline{\bbQ}\,:\,q\textrm{ has degree }\leq d\textrm{ and }G_q\textrm{ is not freely generated by }A\textrm{ and }B_q\right\},
\]
then density and other related questions for this set appear to be unknown. Since the methods of this paper are particularly suited towards understanding accumulation points of non-free numbers in relatively low degree, we ask 

\begin{question}
\label{question:algebraic}
    Is $\mathcal{R}_{\overline{\bbQ}}^{(d)}$ dense in $[-4,4]$ for any $d<\infty$? Does its closure contain any interval or have positive measure?
\end{question}

Lastly, we note that the methods of this paper produce accumulation points of rational non-free numbers whose associated groups, $G_q$, contain relations in which the number of syllables is bounded. The following question, suggested to us by Smilga, represents one remaining gap in the current existing methods for constructing non-free rationals.

\begin{question}
\label{question:syllable}
Does there exist a sequence $q_n\in\mathcal{R}_\bbQ$ such that any relation in the group $G_{q_n}$ has at least $n$ syllables?
\end{question}

Note that if there exists a sequence $(q_n)$ of rational non-free numbers with $q_n\to \pm 4$, then by work in \cite{Yaari_2025}, a subsequence of $(q_n)$ would provide an answer to this question.

\section*{Funding Statement}
During the preparation of this work, the authors received funding from the \textit{Dynamical Systems: Classical, Modern and Quantum} Research Training Grant, funded by NSF grant DMS-2136217.

\section*{Acknowledgements}
This work was initially prepared during the Northwestern University Dynamical Systems REU as a part of the \textit{Dynamical Systems: Classical, Modern and Quantum} RTG (NSF grant DMS-2136217). We thank Aaron Peterson and the other organizers of the REU for the opportunity as well as the seven other participants for the many mathematical conversations and their insights. We especially thank Apoorva Agarwal and Ryan Chatterjee for their patience and time in discussing the topics of this project, Yuchen Liu for an insightful conversation that inspired our use of Diophantine geometry and Joseph Silverman for bringing the work in \cite{poulakis-2003} to our attention. We also thank Ilia Smilga for sharing several helpful comments, examples and for suggesting to us the content of Question~\ref{question:syllable} and Rotem Yaari for bringing his work to our attention.

\section{Background: \texorpdfstring{$n$}{n}-step relations and half-relations} 
\label{sec:def}

\subsection{\texorpdfstring{$n$}{n}-step relation numbers}
\label{sec:step}
In \cite{kimkoberda} Kim and Koberda construct examples of many rational relation numbers $q$ by understanding when relations of a specific form in $G_q$ hold.

\begin{definition}
    Let $F_2=\langle a,b\rangle$ be the free group on two generators. A real number $q$ is an \textit{$n$-step relation number} if there exists a non-trivial word of the form
    \[
    w=b^{m_1}a^{m_2}\ldots b^{m_{2k+1}}\in F_2
    \]
    with $k\in[0,n]$, such that $w(A,B_q)$ is upper-triangular.
\end{definition}

Note that a non-free number is clearly an $n$-step relation number for $n$ equal to the length of a nontrivial word in $A$ and $B_q$ that evaluates to the identity. Conversely, if $q\in\bbQ$ is an $n$-step relation number and $w\in F_2$ is a word as in the above definition, it immediately follows that $q\in\mathcal{R}_\bbQ$ since the word $w' = [waw^{-1},a]\in F_2$ is reduced and $w'(A,B_q)=I_2$. This equivalence is originally attributed to \cite{Lyndon_Ullman_1969}.

The equivalence of being a non-free number and being an $n$-step relation number for some $n$ naturally stratifies the set of all non-free numbers into subsets based, roughly, on how hard it is to exhibit a word in $A$ and $B_q$ that is upper-triangular. Kim and Koberda's results are then a result of using geometric, number-theoretic and computational tools to develop a deep understanding of $1$- and $2$-step relation numbers.
Among these results, Kim and Koberda also give explicit forms of some $1$-step relation numbers.

\begin{lemma}\label{lemma:onestep}[Lemma 3.1 in \cite{kimkoberda}]
    For all nonzero integers $r,s, t$, $\frac{r+t}{rst}$ is a $1$-step relation number.
\end{lemma}

Note that these examples also imply that $\frac{1}{n}$ is a $1$-step relation number for all $n$ since $\frac{1}{n}=\frac{r+t}{rst}$ for $(r,s,t)=(n,2,n)$. Additionally, though not explicitly mentioned in \cite{kimkoberda}, these examples, in fact, account for all $1$-step relation numbers.

\begin{proposition}
\label{prop:onestep}
    Let $q\neq 0$ be a $1$-step relation number. Then $q=\frac{r+t}{rst}$ for some nonzero integers $r,s,t$.
\end{proposition}
\begin{proof}
    Being a $1$-step relation number means that there are integers $r,s,t$ and a word 
    \[
    w=b^r a^{-s} b^t\in F_2
    \]
    so that $w(A,B_q)$ is upper-triangular. For a word of this form, we see that
    \[
    w(A,B_q)=\begin{pmatrix}
        -s t q + 1 & -s \\
-r s tq^2 + (r+t) q & -rsq + 1
    \end{pmatrix}.
    \]
    We then see that for $q\neq 0$, the above matrix is upper-triangular if and only if
    \[
    -rstq+r+t=0,
    \]
    hence $q=\frac{r+t}{rst}$.
\end{proof}

This proof, though entirely elementary, illustrates a common theme throughout \cite{Smilga2021} and the remainder of this work: Often times, a specific relation in $G_q$ holding can be reduced to $q$ being a root of a particular polynomial. Thus, if one can understand such a polynomial sufficiently well, one can construct a number of examples of new non-free numbers.

\subsection{Smilga half-relations and polynomials}
In \cite{Smilga2021}, the author's examples of large accumulation points of non-free rationals come from searching for other types of relations in $G_q$ satisfying a particular symmetry, as opposed to words in $A$ and $B_q$ that evaluate to an upper-triangular matrix. Similarly, as in the proof of Proposition~\ref{prop:onestep}, Smilga then shows that such a relation holding in $G_q$, and thus $q$ being non-free, is equivalent to $q$ being the root of a certain polynomial. The relations studied in \cite{Smilga2021}, deemed \textit{half-relations}, are defined as follows.

\begin{definition}[\cite{Smilga2021} Definition 4]
	A tuple $(a_1, \ldots, a_\ell)\in \bbZ^{\ell}$, with $a_i \neq 0$ for all $i$, is a \textit{half-relation for $q$} if

    \begin{enumerate}
        \item for odd $\ell,$ the matrix $M = A^{a_1}B_q^{a_2}\cdots A^{a_\ell}$ satisfies 

        \[q c_{12}(M) - c_{21}(M) = 0;\]

        \item for even $\ell,$ the matrix $M = A^{a_1}B_q^{a_2}\cdots B_q^{a_\ell}$ satisfies 

        \[c_{11}(M) - c_{22}(M) = 0,\]

        where $c_{ij}(M)$ is the $(i, j)$-th coefficient of $M$.
    \end{enumerate}

    If there exists a half relation $(a_1, \ldots, a_\ell)$ for $q$, then we often say that $q$ is a \textit{length $\ell$ half-relation number}, or just \textit{half-relation number}. 
    \end{definition}

    The connection between these tuples of integers and a relation in $G_q$ comes from the fact that if $(a_1,\ldots,a_\ell)\in\bbZ^\ell$ is a half-relation for $q$, then $A^{a_1}B_q^{a_2}A^{a_3}\cdots C^{a_\ell}=B_q^{a_\ell}A^{a_{\ell-1}}B_q^{a_{\ell-2}}\cdots C^{a_1}$ where $C=A$ when $\ell$ is odd or $C=B_q$ when $\ell$ is even, cf. Proposition 5 of \cite{Smilga2021}. In particular, if $(a_1,\ldots,a_\ell)$ is a half-relation for $q$, then $q$ is non-free. Based on this observation, we arrive at new polynomial conditions on $q$ that determine when certain relations in $G_q$ hold.
  
    \begin{definition}\label{def:PHR}
      For an integer $\ell \ge 0$ and $(a_1, \ldots,a_\ell)\in\bbZ^\ell$ with $a_i\neq 0$, let $M=A^{a_1}B_q^{a_2}\cdots C^{a_\ell}$ where $C=A$ when $\ell$ is odd and $C=B_q$ when $\ell$ is even. Define the polynomial in $q$, $P_{HR}^\ell(a_1,\ldots,a_\ell;q)$ by:
      \[
      P^\ell_{HR}(a_1,\ldots,a_\ell; q)=\begin{cases}
          c_{12}(M)-\frac{1}{q}c_{21}(M) & \ell\textrm{\ is odd},\\
          \frac{1}{q}(c_{11}(M)-c_{22}(M)) & \ell\textrm{\ is even},
      \end{cases}
      \]
    where $c_{ij}(M)$ is the $(i, j)$-th coefficient of $M$. 
    \end{definition}

    Then, we have the following:

    \begin{proposition} For $(a_1,\ldots,a_\ell)\in\bbZ^\ell$ with $a_i \neq 0$,  the following are equivalent:\begin{enumerate}
            \item $(a_1, \ldots, a_\ell)$ is a half-relation for $q$,
            \item $P_{HR}^\ell(a_1, \ldots, a_\ell; q) = 0$.
        \end{enumerate}
\end{proposition}

Smilga's examples of large accumulation points of non-free rational numbers come from studying roots of the above polynomials. In \S\ref{PellSection}, we also use a similar strategy to record other accumulations of non-free rational numbers, such as $1+\sqrt{2}$. 

Having discussed both the $\ell$-step relation numbers of \cite{kimkoberda} and the half-relation numbers of \cite{Smilga2021}, Theorem~\ref{thm:main2} can then be more precisely stated as follows.

\begin{theorem}
\label{thm:main}
All $1$-step relation numbers are accumulations of non-constant sequences of length-$5$ rational half-relation numbers.
\end{theorem}

\section{Diophantine geometry and half-relations}\label{DiophantineSection}

To prove Theorem~\ref{thm:main}, we apply a result in Diophantine geometry to appropriate half-relation polynomials, $P_{HR}^5(a_1,\ldots,a_5;q)$, whose roots limit to $1$-step relation numbers. Specifically, the polynomials $P_{HR}^5(a_1,\ldots,a_5;q)$ are quadratic in $q$ and, solving for their roots, we find many $a_1,\ldots,a_5$ so that the discriminant is a square integer and so that the roots accumulate at $1$-step relation numbers. We note that considering other quadratic polynomial conditions such as being a length $6$ half-relation number yielded no significantly new examples. For larger length half-relation that might correspond to $n$-step relation numbers, the Diophantine obstruction becomes much harder to parse.

\begin{definition}
\label{def:curve}
Let $a = (a_{1}, \dots, a_{5}) \in \mathbb{Z}^{5}$. We define a curve $C_a$ as follows. The discriminant of $P_{\text{HR}}^{5}(a_{1}, \dots, a_{5}; q)$ is
\begin{align*}
	\Delta_{a} & = (a_{3} a_{4} a_{5})^{2} + (2 a_{3} a_{4}^{2} a_{5}^{2} - 2 a_{2} a_{3} a_{4} a_{5}^{2} - 2 a_{2} a_{3}^{2} a_{4} a_{5}) a_{1} \\
		   & + ((a_{2} a_{3})^{2} + (a_{2} a_{5})^{2} + 2 a_{2} a_{4} a_{5}^{2} + 2 a_{2}^{2} a_{3} a_{5} - 2 a_{2} a_{3} a_{4} a_{5} + (a_{4} a_{5})^{2}) a_{1}^{2}
\end{align*}
If we consider $a_2, a_3, a_4, a_5$ to be fixed and $a_1$ to be variable, then the discriminant $\Delta_{a}$ is a square integer if and only if there exists an integer $y$ such that $(x, y) = (a_1, y)$ is a root of a polynomial of the form
\[
	F_{a}(x, y) = \alpha_{a} x^{2} + \beta_{a} x + \gamma_{a} - y^{2}.
\]
We define $C_{a}$ to be the zero locus of $F_{a}$.
\end{definition}

Note that the definitions of $C_a$, $F_a$, $\alpha_a$, $\beta_a$, and $\gamma_a$ all only depend on $a_2, a_3, a_4, a_5$, not on $a_1$. However, we still use the subscript $a$ for brevity.

In the proof of \cref{thm:main}, we use the fact that most of these $C_{a}$ have infinitely many integer points. Precisely,

\begin{lemma} \label{lem:dioph}
	For each non-zero  $a_{3}, a_{4}, a_{5} \in \mathbb{Z}$ with $a_{3} + a_{5} \neq 0$ there are infinitely many $a_{1}, a_{2} \in \mathbb{Z}$ such that $C_{a}$ has infinitely many integer points.
\end{lemma}

We defer proof of this lemma (by Diophantine geometry) until after we use it to prove \cref{thm:main}.

\begin{proof}[Proof of \cref{thm:main}] 
Let $a = (a_{1}, \dots, a_{5}) \in \mathbb{Z}^{5}$. The length $5$ half-relation numbers are exactly the roots $q_{a}^{+}, q_{a}^{-}$ of the polynomials
\begin{align*}
	P_{\text{HR}}^5(a_1, a_2, a_3, a_4, a_5; q) & =  (a_1a_2a_3a_4a_5)q^2+(a_1a_2a_3-a_2a_3a_4+a_1a_2a_5+a_1a_4a_5+a_3a_4a_5)q\\
     & +(a_1-a_2+a_3-a_4+a_5) \\
     & = c_{2} q^{2} + c_{1} q + c_{0} 
\end{align*}
for $a \in \mathbb{Z}^{5}$. Specifically, let
\begin{align*}
	q_{a}^{+} & = \frac{- c_{1} + \sqrt{c_{1}^{2} - 4 c_{0} c_{2}}}{2 c_{0}} & q_{a}^{-} & = \frac{- c_{1} - \sqrt{c_{1}^{2} - 4 c_{0} c_{2}}}{2 c_{0}}
\end{align*}
The roots $q_{a}^{\pm}$ are rational exactly when the discriminant of $P_{\text{HR}}^{5}(a_{1}, \dots, a_{5}; q)$ is a square integer. 

\cref{def:curve} expresses the condition that $\Delta_{a}$ is a square integer as an integer point on the curve $C_{a}$. In \cref{lem:dioph}, we show that for each $a_{3}, a_{4}, a_{5} \in \mathbb{Z}$ with $a_{3} + a_{5} \neq 0$, there are infinitely many $a_{1}, a_{2} \in \mathbb{Z}$ such that $C_{a}$ has infinitely many integer points. We claim that this suffices to prove \cref{thm:main}. 

For each $a_{3}, a_{4}, a_{5}$ with $a_{3} + a_{5} \neq 0$, there are infinitely many rational roots $q_{a}^{-}$ of $P_{\text{HR}}^{5}(a_{1}, \dots, a_{5}; q)$. Then the limit of the rational $q_{a}^{-}$ agrees with the limit of all $q_{a}^{-}$. Specifically, each
\begin{align*}
	\lim_{a_1\rightarrow \infty} q_a^{\pm} & = \frac{-\big(a_2a_3+a_2a_5+a_4a_5\big) \pm \sqrt{(a_2a_3)^2+(a_2a_5)^2+(a_4a_5)^2+2a_2a_4a_5^2+2a_2^2a_3a_5-2a_2a_3a_4a_5}}{2a_2a_3a_4a_5}
\end{align*}
and each
\begin{align*}
	\lim_{a_{2} \to \infty} \lim_{a_{1} \to \infty} q_{a}^{-} & = \frac{-(a_{3} + a_{5})}{a_{3} a_{4} a_{5}} & \lim_{a_{2} \to \infty} \lim_{a_{1} \to \infty} q_{a}^{+} & = 0
\end{align*}
is a limit of rational length $5$ half relation numbers. Note, the condition that $a_{3} + a_{5} \neq 0$ prevents us from obtaining $0$ as a limit of rational $q_{a}^{-}$s. However, we can always obtain $0$ as a limit of rational $q_{a}^{+}$s. Thus, these encompass all 1-step relation numbers of Kim--Koberda.

Note also that even on taking further limits, we only get $1$-step relation numbers as limits --- we have
\[
	\lim_{a_{3} \to \infty} \lim_{a_{2} \to \infty} \lim_{a_{1} \to \infty} q_{a}^{-} =  \frac{-1}{a_{4} a_{5}}
\]

Here, we chose a specific order in which to take limits. However, this order is not necessary to obtain the $1$-step relation numbers of Kim--Koberda as limit points.
\end{proof}

\begin{remark}
    After taking the limit in \(a_1\), we already obtain a family of limit points. In \cite{Smilga2021}, $\pm\frac{3+\sqrt{5}}{2}$ is shown to be a limit point of this form. Specifically, it is an accumulation point of the half-relation numbers associated to \(a=(1, -1, 1, -1, N)\) for appropriate $N$ and is the maximum limit point for length 5 half-relation numbers, up to permutation of \(a_i\)'s. 
\end{remark}

To prove \cref{lem:dioph}, we require conditions under which a conic curve like $C_{a}$ has infinitely many integer points. We use a special case of Theorem A from \cite{SILVERMAN2000163}. We note that this theorem, as stated in the paper, required a small correction found in \cite{poulakis-2003}. We thank Joseph Silverman for pointing this nuance out to us during the preparation of this work. The statement below is equivalent to the corrected statement found in \cite{poulakis-2003}.

\begin{theorem}\label{silverman}
Let \(C\) be a curve defined by \(\alpha x^2 + \beta x+ \gamma-y^2=0\). The set of integer points of $C$, $C(\mathbb{Z})$ has infinitely many points if $C(\mathbb{Z})$ contains at least one non-singular point and either
\begin{enumerate}
	\item the set of points at infinity, $C_{\infty}$ consists of one non-singular rational point

	\item $C_{\infty}$ consists of two non-singular points which are conjugate over a real quadratic field
\end{enumerate}
\end{theorem}

\begin{proof}[Proof of \cref{lem:dioph}]
	Fix non-zero $a_{3}, a_{4}, a_{5} \in \mathbb{Z}$. To show that there are infinitely many $a_{1}, a_{2} \in \mathbb{Z}$ such that $C_{a}$ has infinitely many integer points, we show that there are infinitely many $a_{1}, a_{2}$ such that $C_{a}$ satisfies the hypotheses of \cref{silverman}.

First, we construct a non-singular point in each $C_{a}(\mathbb Z)$. Any singular point $(x, y) \in C_{a}$ must satisfy the following equations (in addition to $F_{a}(x, y) = 0$).
\begin{align*}
	\frac{\partial F_{a}}{\partial x} & = 0  && \iff & x & = -\frac{\beta}{2 \alpha} \\
	\frac{\partial F_{a}}{\partial y} & = 0 & & \iff & y & =0
\end{align*}
Consider the point $(0, a_{3} a_{4} a_{5})$. It's easy to verify $(0, a_{3} a_{4} a_{5}) \in C_{a}(\mathbb Z)$, and yet for non-zero  $a_{3}, a_{4}, a_{5}$, its partial ${\partial F_{a}}/{\partial y}(0, a_{3} a_{4} a_{5}) = a_{3} a_{4} a_{5} \neq 0$. Thus, each $C_{a}(\mathbb Z)$ has a non-singular point.

We compute the projective coordinates for the set of points at infinity \(C_{a,\infty}\), by homogenizing $F_{a}$ to
\[
	F_a^{1}(x : y : z) = \alpha_{a} x^{2} + \beta_{a} x z + \gamma_{a} z^{2} - y^{2}
\]
The points at infinity are exactly those points $(x : y : 0)$ with $F_{a}^{1}(x : y : 0) = 0$. We get
\begin{align*}
	\alpha_a x^2 & = y^2 \\
	y & = \pm x\sqrt{(a_2a_3)^2+(a_2a_5)^2+(a_4a_5)^2+2a_2a_4a_5^2+2a_2^2a_3a_5-2a_2a_3a_4a_5}.
\end{align*}

We claim these points are all non-singular. Notice that $\partial F^{1}_{a} / \partial y = 2y$. If $z = 0$ and $\partial F^{1}_{a} / \partial y = 0$, we have $y = 0$, and thus, $x = \pm \sqrt{\alpha_{a}} y = 0$. However, $(0 : 0 : 0)$ is not a point. Thus, all $(x : y : z)$ on $C_{a, \infty}$ are non-singular.

When \(\alpha_a\) is zero, we have one point at infinity, namely \((1 : 0 : 0)\). When \(\alpha_a\) is strictly positive and not a square integer, $C_a$ has two points at infinity that are conjugate over a real quadratic field. Thus,  \(C_a\) satisfies the conditions of \cref{silverman} except when $\alpha_a < 0$ or a square. We claim that for each $a_{3}, a_{4}, a_{5}$ with $a_{3} + a_{5} \neq 0$ we have $\alpha_{a} > 0$ and non-square for infinitely many $a_{2}$.

Writing $\alpha_a$ as a function of $a_{2}$ we obtain
\[
	\alpha_a = f(a_2) = (a_{3} + a_{5})^{2} a_{2}^{2} + 2 a_{4} a_{5} (a_{5} - a_{3}) a_{2} + (a_{4} a_{5})^{2} =  (a_{3} + a_{5})^{2} a_{2}^{2} + O(a_{2}).
\]

Thus, if $a_{3} + a_{5} \neq 0$, for sufficiently large $a_{2}$ we have $\alpha_a > 0$. Suppose $f$ attained non-square values for only finitely many choices of $a_2$. Then for sufficiently large $N$ and all $a_2\geq N$, $f(a_2)$ is a perfect square. Note that, from the formula for $f$,
\[
\lim_{a_2\to\infty}\sqrt{f(a_2+1)}-\sqrt{f(a_2)}=|a_3+a_5|.
\]
Since the values of $\sqrt{f(a_2+1)}-\sqrt{f(a_2)}$ are always integers for sufficiently large $a_2$, this sequence is eventually constant, and so, we must have that $\sqrt{f(a_2+1)}-\sqrt{f(a_2)}=|a_3+a_5|$ for all sufficiently large $a_2$. In other words, for all sufficiently large $a_2$,
\[
\sqrt{f(a_2)}=|a_3+a_5|a_2 + L
\]
for some constant $L$. Squaring both sides, we have that
\[
\alpha_a = f(a_2) = (|a_3+a_5|a_2+L)^2
\]
for all sufficiently large $a_2$, but as both sides are polynomials agreeing on infinitely many values of $a_2$, this equality must hold for all $a_2$. However, since $a_3,a_4,a_5$ are all nonzero, $\alpha_a$ is not the square of a polynomial in $a_2$ reaching a contradiction. Thus there are always infinitely many $a_{2}$ such that $\alpha_{a}$ is not a square integer and so, for each $a_{3}, a_{4}, a_{5} \in \mathbb{Z}$ with $a_{3} + a_{5} \neq 0$, there are infinitely many $a_{2}$ such that $C_{a}(\mathbb{Z})$ is infinite.
\end{proof}

With an eye towards Question~\ref{question:algebraic}, we note that while the Diophantine obstructions become much harder to parse in larger degree, these methods can be used to establish many other examples of limit points of \textit{algebraic} non-free numbers. For instance, using similar methods, one can show that the largest real roots of the polynomials $P^7_{HR}(1,-1,1,-1,1,N,N;q)$ accumulate at $3$, showing that it is an accumulation point of $\mathcal{R}_{\overline{\bbQ}}^{(3)}$. After the initial posting of this work, Ilia Smilga shared an example of his which shows that $3$ is, in fact, an accumulation point of $\mathcal{R}_\bbQ$. With his permission, we've included it here.

\begin{example}[Smilga]
For all integers $t$, the half-relation polynomials 
\[
P_{HR}^{8}(1,-1,1,-1,1,-(7t+1),-(42t+5)(84t+11),-(252t^2+48t+2))
\]
have $3+\frac{1}{42t+5}$ as a root.
\end{example}

\section{Other Families of Rational Non-free Numbers}

\subsection{Partial Sums of Geometric Series}\label{PartialSumsSubSection}

We will demonstrate through two examples another method to obtain accumulation points of non-free rational numbers by considering partial sums of geometric series. In particular, our first example shows that one-step relation numbers of the form $k^{1-s}$ for positive integers $k, s$, with $s\ge 2$, are accumulation points of one-step relation numbers $k^{1-s} - k^{-t}$, for positive integers $t \ge s$. The numbers $k^{1-s} - k^{-t}$ can be written neatly as a telescoping geometric series, suggesting a larger pattern to be exploited. For example, our second example shows that numbers $(k-1)\sum_{n = 0}^t k^{-s - 2n} =  (k-1)k^{-s}\frac{k^2 - k^{-2t}}{k^2 - 1}$ are non-free (and accumulate at $\frac{k^{2-s}}{k+1}$).  The convergents in both these examples have length 3 half-relation numbers of a similar form.

We came across this form of rational number when considering the base $3$ expansions of the very simplest rational elements of the Cantor middle third set. Note that not all elements of the Cantor set will be relation numbers, in particular, transcendental $q$ are free. Numbers in the Cantor set have a base $3$ expansion that contains only $0$'s and $2$'s. The very simplest of these are those that begin with a block of $0's$, and then a block of $2's$, such as $(.2)_3$ or $(.00022222)_3$. 

\begin{proposition}
    Let $k, t, s$ be positive integers with $t \ge s$ and set $q = k^{1-s} - k^{-t} = (k-1)\sum_{n = s}^tk^{-n}$. Then, we have a length $3$ half-relation $(1, k^{-s + t}(k^s + k), k^{s - 1})$.
\end{proposition}

\begin{example} Before the proof, we demonstrate some examples.
    \begin{itemize}
        \item $q = (.222)_3 = 2/3 + 2/9 + 2/27$. Then $(1, 54, 1)$ is a half-relation.
        \item $q = (.00\ldots 022\ldots2)_3 $ be length $t$ with $(s - 1)$ many $0$'s. Then $(1, 3^{-s + t}(3^s + 3), 3^{s-1})$ is a half-relation.
        \item $q = (.007)_8 = \frac{7}{8^3}$. Then, $(1, 520, 64)$ is a half-relation.
        
    \end{itemize}
\end{example}

\begin{proof}
    
    Recall that $P_{HR}^3(a_1, a_2, a_3; q) = a_1a_2a_3q + a_1 - a_2 + a_3$. We compute that 
    \begin{align*}
        &P_{HR}^3(1, k^{-s + t}(k^s + k), k^{s - 1}; k^{-s + 1} - k^{-t}) \\
        & = k^{-s + t}(k^s + k)(k^{s-1})(k^{-s + 1} - k^{-t}) + 1 - k^{-s + t}(k^s + k) + k^{s - 1} \\
        & =0. 
    \end{align*}
    \end{proof}

We can also consider another simple base $k$ expansion that take a more complicated form than the previous example. Take the expansions that alternate between $0$ and $k-1,$ such as $(.070707)_{8}$, as well as those that start with a block of zeros and then do so, such as $(.00005050)_6$. Precisely, will consider rational numbers of the form \[q = (k-1)\sum_{n = 0}^t k^{-s - 2n} = (k-1)k^{-s}\frac{k^2 - k^{-2t}}{k^2 - 1}\] for positive integers $k, s, t$ and $s\ge 2$. This example shows that numbers $\frac{k^{2-s}}{k+1}$ are accumulation points of non-free rational numbers.

In our previous example, the convergents $k^{1-s} - k^{-t}$ were one-step relation numbers and hence non-free, without the need to exhibit the half-relation of the previous proposition. However, it is not immediately clear that the convergents of our second example $(k-1)k^{-s}\frac{k^2 - k^{-2t}}{k^2 - 1}$ are non-free. We exhibit a length three half-relation of a similar form to our first example. 

    \begin{proposition}
        Let $q = (k-1)k^{-s}\frac{k^2 - k^{-2t}}{k^2 - 1} = (k-1)\sum_{n = 0}^t k^{-s - 2n}$. Let $k, s$ and $t$ be positive integers and $s \ge 2$. Then, we have a length three half-relation $(k^{s-2}, (k^{s-2} + k+1)k^{2t + 2}, k + 1)$.
    \end{proposition}
  
    \begin{example} Before the proof, we demonstrate some examples.
        \begin{itemize}
            \item $q = (.0202)_3$. Then, $ (1, 405, 4)$ is a half-relation
            \item $q = (.000606)_7$. Then, $(49, 136857, 8)$ is a half-relation
        \end{itemize}
    \end{example}

\begin{proof}
First, observe that
\begin{align*}
    &q = (k-1)\sum_{n = 0}^t k^{-s - 2n} = (k-1)k^{-s}\frac{k^2 - k^{-2t}}{k^2 - 1}.
\end{align*}

Now we plug in and verify that $P_{HR}^3(k^{s-2}, (k^{s-2} + k+1)k^{2t + 2}, k + 1; q)$ vanishes. First, compute the leading term:
\begin{align*}
    &(k^{s-2}) ((k^{s-2} + k+1)k^{2t + 2})(k + 1)\left((k-1)k^{-s}\frac{k^2 - k^{-2t}}{k^2 - 1}\right) \\
    &= k^{2t+s} + (k+1)k^{2t + 2} - k^{s-2} - (k+1).
\end{align*}

Then, we can compute the three remaining terms:
\begin{align*}
    & k^{s-2} - (k^{s-2} + k+1)k^{2t + 2} + (k + 1) = k^{s-2} - k^{2t + s} - (k+1)k^{2t + 2} + (k + 1).
\end{align*}

Adding the leading term to the remaining terms gives $0$.
\end{proof}

\subsection{Pell and Half-Pell Sequences} \label{PellSection}

In this section, we will use the techniques of \cite{Smilga2021} to demonstrate two new sequences of non-free numbers that converge to $1 + \sqrt 2$.

Recall the definition of a Pell Number:

    The $n$th Pell Number is the integer that satisfies the recurrence relation
    \begin{align*}
    &P_n = \begin{cases}
        0 & \text{\ if } n = 0 \\
        1 & \text{\ if } n = 1 \\
        2 P_{n-1} + P_{n-2} &\text{\ otherwise } 
    \end{cases}
    \end{align*}.

    The first few terms are $0, 1, 2, 5, 12, 29, 70, 169, \ldots$.

    and the half-companion Pell Numbers that satisfy 

    \begin{align*}
    &H_n = \begin{cases}
        1 & \text{\ if } n = 0 \\
        1 & \text{\ if } n = 1 \\
        2 H_{n-1} + H_{n-2} &\text{\ otherwise } 
    \end{cases}
    \end{align*}.

    The first few terms are $1, 1, 3, 7, 17, 41, 99, 239, \ldots$.

    Observe that $P_n = \frac{(1 + \sqrt2)^n + (1 - \sqrt{2})^n}{2 \sqrt{2}}$. Since $|1 - \sqrt{2}| < 1$, then \[\frac{P_{n+1}}{P_n}\to 1 + \sqrt{2} \text { as } n\to \infty.\]

    Similarly, $H_n = \frac{(1 + \sqrt{2})^n - (1 - \sqrt{2})^n}{2}$ so that \[\frac{H_{n+1}}{H_n}\to 1 + \sqrt{2} \text { as } n\to \infty.\]

    \begin{proposition}
        For $x_n = (-1)^{n+1} 2 P_n P_{n-1}$, then $(1, x_n, 1, -1, 1, -1)$ is a half-relation for $P_{n + 1}/ P_n$. For $y_n = (-1)^n H_n H_{n-1}$, then $(1, y_n, 1, -1, 1, -1)$ is a half-relation for $H_{n+1} / H_n$.
    \end{proposition}

    \begin{example} Let $q_n = P_{n+1}/P_n$ and $a_n = H_{n + 1} / H_n$. The first two terms of each are
    \begin{align*}
        & q_2 = 5/2, x_2 = -4 &  a_2 = 7/3, y_2 = 3 \\
        & q_3 = 12/5, x_3 = 20 &  a_3 = 17/7, y_3 = -21
    \end{align*}
    \end{example}

\begin{proof}
         Let $q_n = P_{n+1}/P_n$ and $a_n = H_{n + 1} / H_n$. Observe that \[P_{HR}^6(1, x_n, 1, -1, 1 ,-1; q_n) = q_{n}^2x_n - 2q_nx_n + 2q_n - x_n - 4 =: f(q_n, x_n).\] We must show that $f(q_n, x_n)$ vanishes for all $n$. Recall the identity $P_{n-1}P_{n+1} - P_{n}^2 = (-1)^n$. Observe that
        \begin{align*}
		P_nf(q_n, x_n) & = 2(-1)^{n + 1} P_{n-1}P_{n+1}^2 - 4(-1)^{n + 1} P_{n-1}P_nP_{n+1} + 2P_{n+1} - 2(-1)^{n + 1} P_{n-1}P_n^2 - 4P_n \\
            & = 2(-1)^{2n + 1}P_{n-1} + 2P_{n + 1} - 4P_n \\
            & = 0.
        \end{align*}

        We will now perform a similar computation for the half companion Pell Numbers, instead plugging in $y_n$ and $a_n$ and using the identity $H_{n+1}H_{n-1} - H_n^2 = 2 (-1)^{n+1}$. Observe that
        \begin{align*}
            H_nf(a_n, y_n) & = (-1)^n H_{n+1}^2 H_{n-1}-(-1)^n 2 H_{n-1}H_n H_{n+1} + 2H_{n+1} - (-1)^n H_n^2 H_{n-1} - 4H_n \\
            & = (-1)^{2n + 1}2 H_{n-1} + 2H_{n+1} - 4H_n \\
            & = 0
        \end{align*}
\end{proof}
\bibliographystyle{alpha}
\bibliography{references}

@article{brenner1955,
  author   = {Joel Lee Brenner},
  title    = {Quelques groupes libres de matrices},
  journal  = {Comptes Rendus de l'Académie des Sciences de Paris},
  volume   = {241},
  year     = {1955},
  pages    = {1689--1691},
  language = {French}
}

@misc{brody-fisher-mj-limbeek-2023,
  title         = {{Greenberg-Shalom's Commensurator Hypothesis and Applications}},
  author        = {Nic Brody and David Fisher and Mahan Mj and Wouter van Limbeek},
  year          = {2023},
  eprint        = {2308.07785},
  archiveprefix = {arXiv},
  primaryclass  = {math.GR},
  url           = {https://arxiv.org/abs/2308.07785}
}

@article{cjr,
  author     = {Chang, Bomshik and Jennings, S. A. and Ree, Rimhak},
  title      = {On certain pairs of matrices which generate free groups},
  journal    = {Canadian J. Math.},
  fjournal   = {Canadian Journal of Mathematics. Journal Canadien de
                Math\'ematiques},
  volume     = {10},
  year       = {1958},
  pages      = {279--284},
  issn       = {0008-414X,1496-4279},
  mrclass    = {20.00},
  mrnumber   = {94388},
  mrreviewer = {J.\ L.\ Brenner},
  doi        = {10.4153/CJM-1958-029-2},
  url        = {https://doi.org/10.4153/CJM-1958-029-2}
}

@article{kimkoberda,
  author  = {Kim, Sang-hyun and Koberda, Thomas},
  year    = {2021},
  month   = {01},
  pages   = {},
  title   = {Non-freeness of Groups Generated by Two Parabolic Elements with Small Rational Parameters},
  volume  = {-1},
  journal = {Michigan Mathematical Journal},
  doi     = {10.1307/mmj/20205868}
}

@article{Lyndon_Ullman_1969,
  title   = {Groups Generated by two Parabolic Linear Fractional Transformations},
  volume  = {21},
  doi     = {10.4153/CJM-1969-153-1},
  journal = {Canadian Journal of Mathematics},
  author  = {Lyndon, R. C. and Ullman, J. L.},
  year    = {1969},
  pages   = {1388–1403}
}

@misc{nybergbrodda,
  title         = {On congruence subgroups of $\operatorname{SL}_2(\mathbb{Z}[\frac{1}{p}])$ generated by two parabolic elements},
  author        = {Carl-Fredrik Nyberg-Brodda},
  year          = {2024},
  eprint        = {2312.11258},
  archiveprefix = {arXiv},
  primaryclass  = {math.GR},
  url           = {https://arxiv.org/abs/2312.11258}
}

@article{poulakis-2003,
  author     = {Poulakis, Dimitrios},
  title      = {Affine curves with infinitely many integral points},
  journal    = {Proc. Amer. Math. Soc.},
  fjournal   = {Proceedings of the American Mathematical Society},
  volume     = {131},
  year       = {2003},
  number     = {5},
  pages      = {1357--1359},
  issn       = {0002-9939,1088-6826},
  mrclass    = {11G30 (11D41 14G25)},
  mrnumber   = {1949864},
  mrreviewer = {Joseph\ H.\ Silverman},
  doi        = {10.1090/S0002-9939-02-06841-7},
  url        = {https://doi.org/10.1090/S0002-9939-02-06841-7}
}

@article{SILVERMAN2000163,
  title    = {On the distribution of integer points on curves of genus zero},
  journal  = {Theoretical Computer Science},
  volume   = {235},
  number   = {1},
  pages    = {163-170},
  year     = {2000},
  issn     = {0304-3975},
  doi      = {https://doi.org/10.1016/S0304-3975(99)00189-9},
  url      = {https://www.sciencedirect.com/science/article/pii/S0304397599001899},
  author   = {Joseph H. Silverman}
}

@article{Smilga2021,
  author  = {Ilia Smilga},
  title   = {New sequences of non-free rational points},
  journal = {Comptes Rendus. Mathématique},
  volume  = {359},
  number  = {8},
  year    = {2021},
  pages   = {983--989},
  doi     = {10.5802/crmath.230},
  url     = {https://www.numdam.org/articles/10.5802/crmath.230/}
}

@article{Yaari_2025,
  title     = {Long relators in groups generated by two parabolic elements},
  issn      = {1661-7215},
  doi       = {10.4171/ggd/899},
  journal   = {Groups, Geometry, and Dynamics},
  publisher = {European Mathematical Society - EMS - Publishing House GmbH},
  author    = {Yaari, Rotem},
  year      = {2025},
  month     = may
}

\end{document}